\documentclass[11pt]{article}
\usepackage{graphicx, colordvi}
\usepackage{dsfont}
\usepackage{epsfig}
\usepackage{mathrsfs}
\usepackage{amssymb,amsfonts,amsmath,amsthm,cite,color}
\usepackage{dsfont}
\usepackage{CJKutf8}
\usepackage{epsfig}
\usepackage{mathrsfs}
\usepackage{longtable}
\usepackage{enumerate}
\usepackage{geometry}
\usepackage{hyperref}
\hypersetup{
	colorlinks=true,
	linkcolor=cyan,
	filecolor=blue,
	urlcolor=red,
	citecolor=green,
}

\newtheorem{theo}{Theorem}

\newtheorem{lem}[theo]{Lemma}

\makeatletter \@addtoreset{equation}{section}
\@addtoreset{theo}{section} \makeatother

\newcommand{\bN} { {\mathbb{N}}}

\newcommand{\bZ} { {\mathbb{Z}}}

\def\qed{\hfill \rule{4pt}{7pt}}

\author{}
\title{}

\begin{document}
\begin{center}
	{\large \bf Congruences involving Delannoy numbers
		and Schröder numbers}
\end{center}

\begin{center}
	{  Chen-Bo Jia}$^{1}$  and{  Jia-Qing Huang}$^{2}$

	$^1$School of Mathematical Sciences\\
	Tiangong University \\
	Tianjin 300387, P.R. China \\
	2230111379@tiangong.edu.cn \\[10pt]
	
	$^2$School of Mathematics\\
	Nanjing University \\
	Nanjing 210093, P.R. China \\
502022210016@smail.nju.edu.cn \\[10pt]
\end{center}

\vskip 6mm \noindent {\bf Abstract.}
The central Delannoy numbers $D_n=\sum_{k=0}^{n}\binom{n}{k}\binom{n+k}{k}$ and the little Schröder number $s_n=\sum_{k=1}^{n}\frac{1}{n}\binom{n}{k}\binom{n}{k-1}2^{n-k}$ are important quantities. In this paper, we confirm
\[\frac{2}{3n(n+1)}\sum_{k=1}^n (-1)^{n-k}k^2D_kD_{k-1}\ \text{and}\ \ \frac 1n\sum_{k=1}^n (-1)^{n-k}(4k^2+2k-1)D_{k-1}s_k\]are positive odd integers for all $n=1,2,3,\cdots$. We also show that for any prime number $p>3$,
\[\sum_{k=1}^{p-1} (-1)^kk^2D_kD_{k-1}\ \equiv\ -\frac56p \pmod{p^2}\]
and
\[\sum_{k=1}^p (-1)^k(4k^2+2k-1)D_{k-1}s_k\ \equiv\ -4p \pmod{p^2}\text{.}\]
 Moreover, define
  \begin{equation*}
s_n(x)=\sum_{k=1}^{n}\frac{1}{n}\binom{n}{k}\binom{n}{k-1}x^{k-1}(x+1)^{n-k},
 \end{equation*}
 for any $n\in\bZ^+$ is even we have  
\begin{equation*}
	\frac{4}{n(n+1)(n+2)(1+2x)^3}\sum_{k=1}^{n}k(k+1)(k+2)s_k(x)s_{k+1}(x)\in\bZ[x].
\end{equation*}


\section{Introduction}
The $n$-th central Delannoy numbers are defined by
\begin{equation*}
	D_{n}=\sum\limits_{k=0}^{n}\binom{n}{k}\binom{n+k}{k}=\sum_{k=0}^{n}\binom{n+k}{2k}\binom{2k}{k} \ \text{for}\ n\in\bN.
\end{equation*}
Such numbers also arise in many enumeration problems in combinatorics\cite{The OEIS Foundation Inc}; for example, $D_n$ is the number of lattice paths from the point (0, 0) to $(n, n)$ with steps $(1, 0), (0, 1)$ and (1, 1).

For $b,c\in\bZ$ and $n\in\bN$, the generalized central trinomial coefficients $T_{n}(b,c)$\cite{Sun2014} is given by
\begin{equation}\label{def:T_{n}}
	T_n(b,c)=\sum_{k=0}^{\lfloor\frac{n}{2} \rfloor}\binom{n}{2k}\binom{2k}{k}b^{n-2k}c^{k}.
\end{equation}
In 2014\cite{Sun2014}, Z.-W. Sun proved 
\begin{equation*}
	\sum_{k=0}^{n-1}(2k+1)T_k(b,c)^2(b^2-4c)^{n-1-k}\equiv 0 \pmod {n^2}
\end{equation*}
for each $n=1,2,3,\cdots$.
Also, $T_{n}(3,2)$ coincides with the central Delannoy number in combinatorics\cite{Sun2018}.

For $n\in\bN$, the $n$-th large Schröder number $S_n$ and little Schröder number $s_n$ are defined
by
\begin{equation*}
	S_n=\sum_{k=0}^{n}\binom{n}{k}\binom{n+k}{k}\frac{1}{k+1}\quad \text{and} \quad 	s_{n}=\sum\limits_{k=1}^{n}N(n,k)2^{n-k}
\end{equation*}
with the Narayana number $N(n,k)$ defined by
\begin{equation*}
	N(n,k)=\frac{1}{n}\binom{n}{k}\binom{n}{k-1}\in\bZ.
\end{equation*}
It is well-known that $S_n = 2s_n$ for $n\in\bZ^+$. There are many combinatorial objects that are counted by $S_n$ or $s_n$. One can consult
items A006318 and A001003 in the OEIS \cite{The OEIS Foundation Inc} for some of their classical combinatorial interpretations.

 Z.-W. Sun \cite{ Sun2012, Sun2014, Sun2018} defined Delannoy polynomials $D_{n}(x),$  the large Schröder
 polynomials $S_n(x)$ and the little Schröder polynomials $s_n(x)$ as
 \begin{equation*}
 	D_{n}(x)=\sum\limits_{k=0}^{n}\binom{n}{k}\binom{n+k}{k}x^k,
 \end{equation*}
 \begin{equation*}
 	S_n(x)=\sum_{k=0}^{n}\binom{n}{k}\binom{n+k}{k}\frac{x^k}{k+1}\quad \text{and} \quad 	s_{n}(x)=\sum\limits_{k=1}^{n}N(n,k)x^{k-1}(x+1)^{n-k}.
 \end{equation*}
Clearly, $D_{n}(1)$, $S_n(1)$ and $s_n(1)$ reduce to the $n$-th central Delannoy numbers, the $n$-th large and the little Schröder numbers respectively. In 2018, Z.-W. Sun \cite[lemma 2.1]{Sun2018} proved that 
\begin{equation}\label{eq: D_{n+1}(x)-D_{n-1}(x)}
	D_{n+1}(x)-D_{n-1}(x)=2x(2n+1)S_{n}(x)
\end{equation}
and
\begin{equation}\label{eq: (x+1)s_{n}(x)}
   (x+1)s_{n}(x)=S_{n}(x),
\end{equation}
for any $n\in\bZ^+$. Sun also proved that
\begin{align*}
	\frac{1}{n}\sum_{k=0}^{n-1}D_{k}(x)s_{k+1}(x)\in\bZ[x(x+1)] \quad\text{for all $n=1,2,3,...$}
\end{align*}
 Arithmetic properties of Delannoy numbers and Schröder numbers are also extensively investigated;
 see \cite{Cao2017,Guo2012a,Guo2012b,Liu2016,Liu2018,Sun2011,Sun2012,Sun2014,Sun2018,Jia2023,Wang2023a,Wang 2023b}. For exmaple, in \cite{Sun2011}, Sun showed that for any odd prime $p$,
 \begin{equation*}
 	\sum_{k=1}^{p-1}\frac{D_k}{k^2}\equiv 2\Bigl(\frac{-1}{p}\Bigr)E_{p-3} \pmod {p} \quad\text{and}\quad \sum_{k=1}^{p-1}\frac{S_k}{6^k}\equiv 0 \pmod {p},
 \end{equation*}
 where $(\frac{\bullet}{p})$ denotes the Legendre symbol and $E_n$ is the $n$-th Euler number. In \cite{Guo2012a}, Guo and Zeng proved 
 \begin{equation*}
 		\sum_{k=0}^{n-1}\varepsilon^k(2k+1)k^r(k+1)^rD_{k}\equiv\sum_{k=0}^{n-1}\varepsilon^k(2k+1)^{2r+1}D_{k}\equiv 0 \pmod {n}, 
 \end{equation*} 
 where $n\geq1$, $r\geq0$, and $\varepsilon=\pm 1$. In \cite{Liu2016}, Liu proved that 
 \begin{equation*}
 	\sum_{k=1}^{p-1}D_kS_k \equiv 2p^3B_{p-3}-2pH^{*}_{p-1} \pmod {p^4} 
 \end{equation*}
 for $p>3$ be a prime, where $B_n$ is the $n$-th Bernoulli number and these $H^{*}_n$ are the alternating harmonic numbers given by $H^{*}_n=\sum_{k=1}^{n}\frac{(-1)^k}{k}$. In \cite{Wang 2023b}, Wang and Zhong proved that for each $r\in\bN$ and prime $p>3$, there is a $p$-adic integer $c_r$ only depending on $r$ such that
 \begin{equation*}
 	\sum_{k=0}^{p-1}(2k+1)^{2r}D_{k} \equiv c_r(\frac{-1}{p}) \pmod {p}.
 \end{equation*}

 In this paper, we confirm some conjectures of Z.-W. Sun \cite{Sun2022}.
 \begin{theo}[Conjecture of Sun\cite{Sun2022}]\label{th:2022sun Remark 5.3(1)}
For any  $n\in\bZ^{+}$,
 	\begin{equation*}
 		\frac{2}{3n(n+1)}\sum_{k=1}^{n}(-1)^{n-k}k^{2}D_{k}D_{k-1}\quad\text{and}\quad	\frac{1}{n}\sum_{k=1}^{n}(-1)^{n-k}(4k^{2}+2k-1)D_{k-1}s_{k}
 	\end{equation*}
 	are positive odd integers.
 \end{theo}
On the basis of Theorem \ref{th:2022sun Remark 5.3(1)}, we prove the following observation by Z.-W. Sun.
 \begin{theo}\label{th:2022sun Remark 5.3(2)}
	For any prime $p\textgreater3,$ we have 
	 \[\sum_{k=1}^{p-1} (-1)^kk^2D_kD_{k-1}\ \equiv\ -\frac56p\pmod{p^2}\text{,}\]
	and
	\[\sum_{k=1}^p (-1)^k(4k^2+2k-1)D_{k-1}s_k\ \equiv\ -4p \pmod{p^2}\text{.}\]
\end{theo}

\begin{theo}[Conjecture of Sun\cite{Sun2022}]\label{th:2022sun}
	Let $n\in\bZ^+$ is even. Then
	\begin{equation*}
		\frac{4}{n(n+1)(n+2)(1+2x)^3}\sum_{k=1}^{n}k(k+1)(k+2)s_k(x)s_{k+1}(x)\in\bZ[x].
	\end{equation*}
\end{theo}
The outline of the article is as follows. In Section 2--3, we  will provide two proof methods for Theorem \ref{th:2022sun Remark 5.3(1)}--\ref{th:2022sun Remark 5.3(2)}. We will prove Theorem \ref{th:2022sun} in Section 4.

\section{Proof of  Theorem \ref{th:2022sun Remark 5.3(1)} and \ref{th:2022sun Remark 5.3(2)} }
Let
\begin{equation}\label{eq:D_{n}}
	D_{n}=T_{n}(3,2)=\sum_{k=0}^{\lfloor\frac{n}{2} \rfloor}\binom{n}{2k}\binom{2k}{k}3^{n-2k}2^{k}.
\end{equation}
To prove Theorem \ref{th:2022sun Remark 5.3(1)}, we first establish two lemmas.

\begin{lem}\label{lem: D_{n}}
	For $n\in\bZ^+$,
	\begin{equation*}
		\frac{D_{n}(D_{n+1}-3D_{n})+D_{n-1}(D_{n}+3D_{n+1})}{54}
	\end{equation*}
	is a positive odd integer, where $D_{n}$ is given by \eqref{eq:D_{n}}.
\end{lem}
\noindent\emph{Proof.}
When $n=1$, we have	
\begin{equation*}
	\frac{D_{1}(D_{2}-3D_{1})+D_{0}(D_{1}+3D_{2})}{54}=1.
\end{equation*}
Next, we assume $n\geq2$. From \eqref{eq:D_{n}}, we get
\begin{align*}
\frac{D_{n+1}-3D_{n}}{2}
	&=\frac{1}{2}\left(\sum_{k=0}^{\lfloor\frac{n+1}{2} \rfloor}\binom{n+1}{2k}\binom{2k}{k}3^{n+1-2k}2^{k}-3\sum_{k=0}^{\lfloor\frac{n}{2} \rfloor}\binom{n}{2k}\binom{2k}{k}3^{n-2k}2^{k}\right)\notag\\
	&=\sum_{k=1}^{\lfloor\frac{n+1}{2} \rfloor}\binom{n+1}{2k}\binom{2k-1}{k-1}3^{n+1-2k}2^{k}-\sum_{k=1}^{\lfloor\frac{n}{2} \rfloor}\binom{n}{2k}\binom{2k-1}{k-1}3^{n+1-2k}2^{k}\notag\\
	&\equiv 0 \pmod{2}\notag
\end{align*}
and 
\begin{align*}
\frac{D_{n}+3D_{n+1}}{2}
	&=\frac{1}{2}\left(\sum_{k=0}^{\lfloor\frac{n}{2} \rfloor}\binom{n}{2k}\binom{2k}{k}3^{n-2k}2^{k}+3\sum_{k=0}^{\lfloor\frac{n+1}{2} \rfloor}\binom{n+1}{2k}\binom{2k}{k}3^{n+1-2k}2^{k}\right)\notag\\
	&=\sum_{k=1}^{\lfloor\frac{n}{2} \rfloor}\binom{n}{2k}\binom{2k-1}{k-1}3^{n-2k}2^{k}+3\sum_{k=1}^{\lfloor\frac{n+1}{2} \rfloor}\binom{n+1}{2k}\binom{2k-1}{k-1}3^{n+1-2k}2^{k}+5\times3^{n}\notag \\
	&\equiv 1 \pmod{2}.\notag
\end{align*}
As $D_n$ is a positive odd, we have
\begin{equation*}
	\frac{D_{n}(D_{n+1}-3D_{n})+D_{n-1}(D_{n}+3D_{n+1})}{2}
\end{equation*}
is a positive odd integer.
By $(2,3)=1$, hence we only need to prove
\begin{equation*}
	D_{n}(D_{n+1}-3D_{n})+D_{n-1}(D_{n}+3D_{n+1})\equiv 0 \pmod {27}.
\end{equation*}
When $n$ is odd,
\begin{equation}\label{eq: odd}
	D_{n}=\sum_{k=0}^{\frac{n-1}{2}}\binom{n}{2k}\binom{2k}{k}3^{n-2k}2^{k}.
\end{equation}
When $n$ is even,
\begin{equation}\label{eq: even}
	D_{n}=\sum_{k=0}^{\frac{n}{2}}\binom{n}{2k}\binom{2k}{k}3^{n-2k}2^{k}.
\end{equation}

Next, we will proceed according to the parity of $n$.
If $n$ is a odd, by \eqref{eq: odd} and \eqref{eq: even}, we have
\begin{align}
	&D_{n}(D_{n+1}-3D_{n})+D_{n-1}(D_{n}+3D_{n+1})\notag\\
	&\equiv 3n\binom{n-1}{\frac{n-1}{2}}\binom{n+1}{\frac{n+1}{2}}2^{n}+3n\binom{n-1}{\frac{n-1}{2}}^{2}2^{n-1}+3\binom{n-1}{\frac{n-1}{2}}\binom{n+1}{\frac{n+1}{2}}2^{n}\notag\\
	&\equiv
	2^{n-1}\binom{n-1}{\frac{n-1}{2}}\Bigg(6n\binom{n+1}{\frac{n+1}{2}}+3n\binom{n-1}{\frac{n-1}{2}}+6\binom{n+1}{\frac{n+1}{2}}\Bigg)\notag\\
	&\equiv
	27(n+1)2^{n-2}\binom{n-1}{\frac{n-1}{2}}\binom{n}{\frac{n-1}{2}}\equiv 0 \pmod{27}.\notag
\end{align}
If $n$ is a even, we can obtain 
\begin{align}
	&D_{n}(D_{n+1}-3D_{n})+D_{n-1}(D_{n}+3D_{n+1})\notag\\
	&\equiv 3n2^{n}\binom{n}{\frac{n}{2}}^{2}+3(n-1)2^{n-1}\binom{n-2}{\frac{n-2}{2}}\binom{n}{\frac{n}{2}}\notag\\
	&\equiv2^{n-1}\binom{n}{\frac{n}{2}}\Bigg(6n\binom{n}{\frac{n}{2}}+3(n-1)\binom{n-2}{\frac{n-2}{2}}\Bigg)\notag\\
	&\equiv
	27(n-1)2^{n-1}\binom{n}{\frac{n}{2}}\binom{n-2}{\frac{n-2}{2}}\equiv 0 \pmod{27}.\notag
\end{align}
This completes the proof.
\qed

\begin{lem}\label{lem: D_{n}s_{n}}
		For $n\in\bZ^+$,
	\begin{equation}\label{eq: mod3}
		(n+1)D_{n}s_{n}+(n+2)D_{n-1}s_{n+1}\equiv 0 \pmod 3.
	\end{equation}
\end{lem}
\noindent\emph{Proof.}
When $n=1$, $2D_{1}s_{1}+3D_{0}s_{2}=15\equiv 0 \pmod 3$. Next we assume $n\geq2$. By Lemma 1.2 of \cite{Sun2018}, we have
\begin{equation*}
	D_{n-1}=D_{n+1}-4(2n+1)s_{n}.
\end{equation*}
and hence
\begin{align}\label{eq: D_{n}s_{n}}
	&(n+1)D_{n}s_{n}+(n+2)D_{n-1}s_{n+1}\notag\\
	&=(n+1)D_{n}s_{n}+(n+2)(D_{n+1}-4(2n+1)s_{n})s_{n+1} \notag\\
	&=(n+1)D_{n}s_{n}+(n+2)D_{n+1}s_{n+1}-4(n+2)(2n+1)s_{n}s_{n+1}.
\end{align}
In light of \cite[Section 4]{Sun2018} and \cite[Lemma 4.5]{Sun2022}, we have $s_{2n}(x)/(1+2x)\in\bZ[x]$ for all $n\in\bZ^{+}$. Combining this with \eqref{eq: odd} and \eqref{eq: D_{n}s_{n}}, we get \eqref{eq: mod3}.
\qed

\noindent\emph{Proof of Theorem \ref{th:2022sun Remark 5.3(1)}.}
Let $a_{k}=(-1)^{n-k}D_{k}$ and $b_{k}=D_{k-1}$, by Zeilberger’s algorithm\cite{Zeilberger1990} we have
\begin{equation}\label{a_{k}=(-1)^{n-k}D_{k}}
(k+2)a_{k+2}=-3(3+2k)a_{k+1}-(k+1)a_{k}
\end{equation}
and
\begin{equation}\label{b_{k}=D_{k-1}}
(k+1)b_{k+2}=3(1+2k)b_{k+1}-kb_{k}.
\end{equation}
By the telescoping method\cite{Hou2023}, we find that
\begin{equation*}
k^{2}a_{k}b_{k}=\Delta_{k}F_{k}
\end{equation*}
where $\Delta$ is the difference operator
(that is, $\Delta F(n) = F(n + 1)-F(n)$) and 
\begin{equation*}
F_{k}=\frac{1}{72}(k-1)a_{k-1}b_{k-1}+\frac{1}{12}k(1-k)a_{k}b_{k-1}+\frac{1}{24}(2k^2-2k+1)a_{k-1}b_{k}+\frac{1}{72}ka_{k}b_{k}.
\end{equation*}
Since $F_{1}=\frac{b_{0}}{72}(3a_{0}+a_{1})=0$, we derive
\begin{equation}\label{k^{2}a_{k}b_{k}3}
\sum_{k=1}^{n}k^{2}a_{k}b_{k}=\frac{1}{72}na_{n}b_{n}-\frac{1}{12}n(n+1)a_{n+1}b_{n}+\frac{1}{24}(2n^2+2n+1)a_{n}b_{n+1}+\frac{1}{72}(n+1)a_{n+1}b_{n+1}.
\end{equation}
Combining equality \eqref{k^{2}a_{k}b_{k}3} and \eqref{b_{k}=D_{k-1}}, we obtain
\begin{equation}\label{eq: k^2D_kD_{k-1}}
\sum_{k=1}^{n}(-1)^{n-k}k^{2}D_{k}D_{k-1}=(n+1)\left(\frac{3(n+1)D_{n}^{2}+(8+18n)D_{n+1}D_{n}-3(n+1)D_{n+1}^2}{36} \right).
\end{equation}
By recurrence relation of $D_{n}$, we also have
\begin{equation*}
D_{n+1}-3D_{n}=n\left(6D_{n}-D_{n+1}-D_{n-1} \right).
\end{equation*} 
Hence, by equality \eqref{eq: k^2D_kD_{k-1}} we deduce that
\begin{align*}
&\frac{36}{n+1}\sum_{k=1}^{n}(-1)^{n-k}k^{2}D_{k}D_{k-1}\notag\\
&=3(n+1)D_{n}^{2}+(8+18n)D_{n+1}D_{n}-3(n+1)D_{n+1}^{2}\notag\\
&=3D_{n}^{2}+8D_{n}D_{n+1}-3D_{n+1}^{2}+n(3D_{n}^{2}+18D_{n}D_{n+1}-3D_{n+1}^{2})\notag\\
&=n(3D_{n}^{2}+18D_{n}D_{n+1}-3D_{n+1}^{2})-(D_{n+1}-3D_{n})(3D_{n+1}+D_{n})\notag\\
&=n\bigl(D_{n}(D_{n+1}-3D_{n})+D_{n-1}(D_{n}+3D_{n+1}) \bigr).
\end{align*}
Therefore, we have
\begin{equation*}
	\frac{2}{3n(n+1)}\sum_{k=1}^{n}(-1)^{n-k}k^{2}D_{k}D_{k-1}=\frac{D_{n}(D_{n+1}-3D_{n})+D_{n-1}(D_{n}+3D_{n+1})}{54}.
\end{equation*}
 Applying Lemma \ref{lem: D_{n}}, we complete the proof of the first half of Theorem \ref{th:2022sun Remark 5.3(1)}.
 
 Similarly, let $a_{k}=D_{k-1}$ and $b_{k}=(-1)^{n-k}s_{k}$. By Zeilberger’s algorithm\cite{Zeilberger1990}, we find that
\begin{equation*}
(k+1)a_{k+2}=3(1+2k)a_{k+1}-ka_{k}\quad\text{and}\quad(k+3)b_{k+2}=-3(3+2k)b_{k+1}-kb_{k}.
\end{equation*}
By the telescoping method\cite{Hou2023}, we find that
\begin{equation*}
(4k^{2}+2k-1)a_{k}b_{k}=\Delta_{k}G_{k}
\end{equation*}
where
\begin{equation*}
G_{k}=\frac{1}{3}(k-1)\big(ka_{k}b_{k-1}-(k+1)a_{k-1}b_{k} \big).
\end{equation*}
Due to $G_{1}=0$, we derive
\begin{equation}
\sum_{k=1}^{n}(4k^{2}+2k-1)a_{k}b_{k}=n\left(\frac{(n+1)a_{n+1}b_{n}-(n+2)a_{n}b_{n+1}}{3} \right).
\end{equation}
Hence, we clearly have
\begin{equation}\label{eq: (-1)^{n-k}D_{k-1}s_{k}}
\frac{1}{n}\sum_{k=1}^{n}(-1)^{n-k}(4k^{2}+2k-1)D_{k-1}s_{k}=\frac{(n+1)D_{n}s_{n}+(n+2)D_{n-1}s_{n+1}}{3}.
\end{equation}
By  $s_{n}$ and $D_{n}$ being positive odd integers, we have $(n+1)D_{n}s_{n}+(n+2)D_{n-1}s_{n+1}$ is a positive odd integer for any $n\in\bZ^+$. This completes the proof with the help of Lemma \ref{lem: D_{n}s_{n}}.
\qed

The following lemma will be used in our proof of Theorem \ref{th:2022sun Remark 5.3(2)}.
\begin{lem}\label{lem: s_n and s_{n+1}}
	For any odd prime p, we have
	\begin{equation*}
		s_p \equiv 2 \pmod{p} \quad\text{and}\quad s_{p+1} \equiv 3 \pmod{p}.
	\end{equation*}
\end{lem}
\noindent\emph{Proof.}
Let $p$ be an odd prime. Then we have  
\begin{align*}
	s_p&=\sum_{k=1}^{p}\frac{1}{p}\binom{p}{k}\binom{p}{k-1}2^{p-k}\\
	&=2^{p-1}+\sum_{k=2}^{p-1}\frac{1}{p}\binom{p}{k}\binom{p}{k-1}2^{p-k}+1\\
	&\equiv 2^{p-1}+1 \equiv 2 \pmod{p},
\end{align*}
and
\begin{align*}
	s_{p+1}=&\sum_{k=1}^{p+1}\frac{1}{p+1}\binom{p+1}{k}\binom{p+1}{k-1}2^{p+1-k}\\
		&=2^{p}+\sum_{k=2}^{p}\frac{1}{p+1}\binom{p+1}{k}\binom{p+1}{k-1}2^{p+1-k}+1\\
	&\equiv 2^{p}+1 \equiv 3 \pmod{p}.
\end{align*}
\qed

\noindent\emph{Proof of Theorem \ref{th:2022sun Remark 5.3(2)}.}
Let $p\textgreater3$ be a prime. By  equality \eqref{eq: k^2D_kD_{k-1}}, we have
\begin{align*}
	\frac{1}{p}\sum_{k=1}^{p-1}k^2D_kD_{k-1}&=\frac{1}{12}pD_{p-1}^2+\frac{1}{2}(p-1)D_{p-1}D_{p}+\frac{2}{9}D_{p-1}D_{p}-\frac{1}{12}pD_{p}^2\\
	&\equiv -\frac{5}{18}D_{p-1}D_{p} \pmod {p}.
\end{align*}
From the proof of \cite[Lemma 4]{Jia2023}, we get
\begin{equation*}
	D_{p-1}\equiv 1 \pmod{p} \quad\text{and}\quad D_{p}\equiv 3 \pmod{p}.
\end{equation*}
Thus, we have
\begin{equation*}
	\frac{1}{p}\sum_{k=1}^{p-1}k^2D_kD_{k-1}\equiv -\frac{5}{6} \pmod{p}.
\end{equation*}
By equality \eqref{eq: (-1)^{n-k}D_{k-1}s_{k}}, we have
\begin{equation*}
\frac{1}{p}\sum_{k=1}^{p}(-1)^k(4k^{2}+2k-1)D_{k-1}s_{k}=-\frac{1}{3}(p+1)D_{p}s_{p}-\frac{1}{3}(p+2)D_{p-1}s_{p+1}.
\end{equation*}
In view of Lemma \ref{lem: s_n and s_{n+1}}, we have 
\begin{equation*}
	\frac{1}{p}\sum_{k=1}^{p}(-1)^k(4k^{2}+2k-1)D_{k-1}s_{k}\equiv -4 \pmod{p}.
\end{equation*}
This ends the proof. \qed

\section{Another proof of  Theorem \ref{th:2022sun Remark 5.3(1)} and \ref{th:2022sun Remark 5.3(2)} }
\begin{lem}\label{lem: nD_nD_{n-1}}
	For any $n\in\bZ^+$, we have 
	 \begin{equation}\label{eq: nD_nD_{n-1}}
		nD_nD_{n-1}\ =\ 3\sum_{j=0}^{n-1}(n-j)\binom{n+j}{2j}{\binom{2j}j}^22^j\text{.}  
	\end{equation}
\end{lem}
\begin{proof}
Taking $b=2$ and $c=1$ in equality (4.1) of \cite{Sun2022} yields.
\end{proof}

\begin{lem}\label{lem: (-1)^{n-k}k(k-j)}
	For any $j\in\bN$ and integer $n\textgreater j$, we have 
	\begin{equation}\label{eq: (-1)^{n-k}k(k-j)}
		\sum_{k=j+1}^n (-1)^{n-k}k(k-j)\binom{k+j}{2j}\ =\ \frac 1{2\binom{2j}j}\binom{n+j+1}j\binom{n-1}jn(n+1)\text{.}
	\end{equation}
\end{lem}
\noindent\emph{Proof.}
When $n=j+1$, 
\begin{align*}
	& \sum_{k=j+1}^n(-1)^{n-k}k(k-j)\binom{k+j}{2j}\quad\quad\quad\quad\quad\quad\quad\\
	={} & (j+1)(2j+1) \notag \\
	={} & \frac1{2\binom{2j}j}\binom{2j+2}j(j+1)(j+2) \notag\\[6pt]
	={} & \frac 1{2\binom{2j}j}\binom{n+j+1}j\binom{n-1}jn(n+1)\text{.}
\end{align*}
Assuming that equation \eqref{eq: (-1)^{n-k}k(k-j)} holds when $n=m(j+1\leq m\in\bZ)$. Then for $n=m+1$, 
 \begin{align*}
	& \sum_{k=j+1}^{n} (-1)^{n-k}k(k-j)\binom{k+j}{2j} \notag\\[6pt]
	={} & (m+1)(m+1-j)\binom{m+1-j}{2j}-\frac 1{2\binom{2j}j}\binom{m+j+1}j\binom{m-1}jm(m+1) \notag\\[6pt]
	={} & \frac{(m+j+1)!}{(2j)!(m-j)!}(m+1)-\frac12 \frac{(m+j+1)!}{(2j)!(m-j-1)!} \notag\\[6pt]
	={} & \frac12\frac{(m+j+2)!}{(2j)!(m-j)!} \notag\\[6pt]
	={} & \frac 1{2\binom{2j}j}\binom{n+j+1}j\binom{n-1}jn(n+1)\text{.}
\end{align*}
In view of the above, by induction, equality \eqref{eq: (-1)^{n-k}k(k-j)} holds.
\qed
\begin{lem}\label{lem: D_{n-1}s_{n}}
	For any $n \in \bZ^+$, we have 
	 \begin{equation}\label{eq: D_{n-1}s_{n}}
		D_{n-1}s_n\ =\ \sum_{j=0}^{n-1}\binom{n+j}{2j}{C_j}^2\left( 2j+1-j(j+1)\frac{2n+1}{n(n+1)} \right)2^j\text{,}   
	\end{equation}
	where $C_j$ denotes the Catalan number $\binom{2j}{j}/(j+1)$.
\end{lem}
\noindent\emph{Proof.}
Taking $x=2$ in equality (2.11) of \cite{Sun2018} yields.
\qed

\begin{lem}\label{lem: (-1)^{n-k}(4k^2+2k-1)}
	 For any $j \in \mathbb{N}$ and integer $n>j$, we have
	\begin{align}\label{eq: (-1)^{n-k}(4k^2+2k-1)}
		& \sum_{k=j+1}^n(-1)^{n-k}(4k^2+2k-1)\binom{k+j}{2j}\left(2j+1-j(j+1)\frac{2k+1}{k(k+1)}\right)\notag\\[6pt]
		={} & \frac1{C_j}w(n,j+1)\ n\ \big((4j+2)n+4j+3\big)\text{,}  
	\end{align}
	where $w(n,k)=\frac{1}{k}\binom{n-1}{k-1}\binom{n+k}{k-1}\in\bZ.$
\end{lem}
\noindent\emph{Proof.}
When $n=j+1$, we have 
\begin{align*}
	& \sum_{k=j+1}^n(-1)^{n-k}(4k^2+2k-1)\binom{k+j}{2j}\left(2j+1-j(j+1)\frac{2k+1}{k(k+1)}\right) \notag\\[7pt]
	={} & \frac{2j+2}{j+2}(4j^2+10j+5)(2j+1) \notag\\[7pt]
	={} & \frac{j+1}{\binom{2j}j}\big((4j+2)(j+1)+4j+3\big)\binom{2j+2}j  \notag\\[7pt]
	={} & \frac{n}{\binom{2j}j}\big((4j+2)n+4j+3\big)\binom{n-1}j\binom{n+j+1}j \notag\\[7pt]
	={} & \frac{1}{C_j}w(n,j+1)\ n\ \big((4j+2)n+4j+3\big)\text{.}
\end{align*}
Now assume that (\ref{eq: (-1)^{n-k}(4k^2+2k-1)}) holds for a fixed positive integer $n=m (j+1\leq m \in\bZ)$. Then for $n=m+1$,
\begin{align*}
	& \sum_{k=j+1}^{n}(-1)^{n-k}(4k^2+2k-1)\binom{k+j}{2j}\left(2j+1-j(j+1)\frac{2k+1}{k(k+1)}\right) \notag\\[7pt]
	={} & \big(4(m+1)^2+2(m+1)-1\big)\binom{m+j+1}{2j}\left(2j+1-j(j+1)\frac{2m+3}{(m+1)(m+2)}\right) \notag\\[7pt]
	{} & -\sum_{k=j+1}^m(-1)^{m-k}(4k^2+2k-1)\binom{k+j}{2j}\left(2j+1-j(j+1)\frac{2k+1}{k(k+1)}\right) \notag\\[7pt]
	={} & \frac{1}{(m+1)(m+2)}\binom{m+j+1}{2j}(m-j-1)(2mj+m+3j+2)(4m^2+m+3j+5) \notag\\[7pt]
	{} & -\frac{n}{\binom{2j}j}\big((4j+2)m+4j+3\big)\binom{m+j+1}j\binom{m-1}j \notag\\[7pt]
	={} & \frac{1}{(m+1)(m+2)}\ \frac{(m+j+1)!}{(2j)!(m-j)!}(m+1)(m+j+2)(4mj+2m+8j+5) \notag\\[7pt]
	={} & \frac{m+1}{\binom{2j}j}\big((4j+2)(m+1)+4j+3\big)\binom{m}j\binom{m+j+2}j \notag\\[7pt]
	={} & \frac{1}{C_j}w(n,j+1)\ n\ \big((4j+2)n+4j+3\big)\text{.}
\end{align*}
According to induction, equality \eqref{eq: (-1)^{n-k}(4k^2+2k-1)} holds.
\qed

\noindent\emph{Another proof of Theorem \ref{th:2022sun Remark 5.3(1)}.}
 Let $\mathcal{A}_n=\frac2{3n(n+1)}\sum_{k=1}^n(-1)^{n-k}k^2D_kD_{k-1}$, then by Lemma \ref{lem: nD_nD_{n-1}} and \ref{lem: (-1)^{n-k}k(k-j)} we have
\begin{align}\label{eq: A_n}
	\mathcal{A}_n &= \frac2{3n(n+1)}\sum_{k=1}^n(-1)^{n-k}k\left(3\sum_{j=0}^{k-1}(k-j)\binom{k+j}{2j}{\binom{2j}j}^22^j\right) \notag\\[4pt]
	&=\frac2{n(n+1)}\sum_{j=0}^{n-1}{\binom{2j}j}^22^j\left(\sum_{k=j+1}^n(-1)^{n-k}k(k-j)\binom{k+j}{2j}\right)\notag\\[4pt]
	&=\frac1{n(n+1)}\sum_{j=0}^{n-1}\binom{2j}j2^j\binom{n+j+1}{j}\binom{n-1}j\ n\ (n+1)\notag\\[4pt]
	&=\sum_{j=0}^{n-1}\binom{2j}j\binom{n+j+1}j\binom{n-1}j2^j\\
	&=1+\sum_{j=1}^{n-1}\binom{2j}j\binom{n+j+1}j\binom{n-1}j2^j \notag\\
	&\equiv \ 1 \pmod{2}\text{.}\notag
\end{align}
 Let $\mathcal{B}_n=\frac 1n\sum_{k=1}^n (-1)^{n-k}(4k^2+2k-1)D_{k-1}s_k$, then by Lemma \ref{lem: D_{n-1}s_{n}} and \ref{lem: (-1)^{n-k}(4k^2+2k-1)} we have
\begin{align}\label{eq: B_n}
	\mathcal{B}_n &= \frac1n\sum_{k=1}^n(-1)^{n-k}(4k^2+2k-1)\sum_{j=0}^{n-1}\binom{n+j}{2j}{C_j}^2\left( 2j+1-j(j+1)\frac{2n+1}{n(n+1)} \right)2^j\notag\\[4pt]
	&=\frac1n\sum_{j=0}^{n-1}{C_j}^22^j\sum_{k=j+1}^n(-1)^{n-k}(4k^2+2k-1)\binom{k+j}{2j}\left(2j+1-j(j+1)\frac{2k+1}{k(k+1)}\right)\notag\\[4pt]
	&=\frac1n\sum_{j=0}^{n-1}{C_j}^22^j\frac1{C_j}w(n,j+1)\ n\ \big((4j+2)n+4j+3\big)\notag\\[4pt]
	&=\sum_{j=0}^{n-1}C_jw(n,j+1)\big((4j+2)n+4j+3\big)2^j\\[4pt]
	&=2n+3+\sum_{j=1}^{n-1}C_jw(n,j+1)\big((4j+2)n+4j+3\big)2^j\notag\\[4pt]
	&\equiv \ 1  \pmod{2}\text{.}\quad\quad\quad\quad\quad\quad\quad\quad\quad\quad\quad\quad\quad\quad\quad\quad\quad\quad\quad\quad\quad\quad\quad\quad\quad\quad\quad\notag
\end{align}
This concludes the proof.
\qed

Now we provide another proof of Theorem \ref{th:2022sun Remark 5.3(2)} with following lemmas.
\begin{lem}\label{lem: (4j+3)C_j}
For any $j \in \mathbb{N}$, we have 
	\[(4j+3)C_j\ =\ 2\binom{2j}j+\binom{2j+1}{j+1}\text{.}\]
\end{lem}
\noindent\emph{Proof.} Clearly,
 \begin{align*}
	(4j+3)C_j&=\frac{(2j)!}{j!(j+1)!}\big((2j+2)+(2j+1)\big) \notag\\[6pt]
	&=2\frac{(2j)!}{j!j!}+\frac{(2j+1)!}{j!(j+1)!}=2\binom{2j}j+\binom{2j+1}{j+1}.
\end{align*}
\qed

\begin{lem}\label{lem: sum}
	For any prime $p>3$,
	\begin{align*}
			\sum_{j=0}^{p-2}\binom{2j}j(-2)^j &\equiv\ 1\ \pmod{p}\text{,}\\
		\sum_{j=0}^{p-2}\binom{2j+1}{j+1}(-2)^j &\equiv\ 0\ \pmod{p}\text{,}\\
		\sum_{j=0}^{p-2}\binom{2j}j\ j\ (-2)^j &\equiv\ -\frac49\ \pmod{p}\text{.}
	\end{align*}
\end{lem}
\noindent\emph{Proof.}
	For any prime $p>3$, 
\begin{align*}
	\sum_{j=0}^{p-2}\binom{2j}j(-2)^j &\equiv \sum_{j=0}^{\frac{p-1}2}\binom{2j}j(-2)^j = \sum_{j=0}^{\frac{p-1}2}\binom{-\frac12}{j}8^j\\
	&\equiv \sum_{j=0}^{\frac{p-1}2}\binom{\frac{p-1}2}{j}8^j = (1+8)^{\frac{p-1}2} = 3^{p-1}\\
	&\equiv 1\pmod{p}\text{,}\\
	\sum_{j=0}^{p-2}\binom{2j+1}{j+1}(-2)^j &= \frac12\sum_{j=1}^{p-1}\binom{2j}j(-2)^{j-1} = -\frac14\sum_{j=1}^{p-1}\binom{2j}j(-2)^j\\[6pt]
	&\equiv -\frac14\left(\sum_{j=0}^{p-2}\binom{2j}j(-2)^j-1\right)\\
	&\equiv 0\pmod{p}\text{,}\\
	\sum_{j=0}^{p-2}\binom{2j}j\ j\ (-2)^j &=(-4)\sum_{j=0}^{p-3}\binom{-\frac23}j8^j \equiv(-4)\sum_{j=0}^{\frac{p-3}2}\binom{-\frac23}j8^j\\
	&\equiv(-4)\sum_{j=0}^{\frac{p-3}2}\binom{\frac{p-3}2}j8^j =(-4)(1+8)^{\frac{p-3}2}=(-4)3^{p-3}\\
	&\equiv -\frac49\pmod{p}\text{.}
\end{align*}
This ends the proof.
\qed

\noindent\emph{Another proof of Theorem \ref{th:2022sun Remark 5.3(2)}.}
For any prime $p>3$, let $\mathcal{A}_{p-1}=\frac2{3(p-1)p}\sum_{k=1}^{p-1}(-1)^kk^2D_kD_{k-1}$ and $\mathcal{B}_{p}=\frac1p\sum_{k=1}^p(-1)^{k+1}(4k^2+2k-1)D_{k-1}s_k$. In view of equality \eqref{eq: A_n} and \eqref{eq: B_n}, we have 
\begin{equation*}
	\mathcal{A}_{p-1} = \sum_{j=0}^{p-2}\binom{2j}j\binom{p+j}j\binom{p-2}j2^j\quad\text{and}\quad\mathcal{B}_{p}= \sum_{j=0}^{p-1}C_jw(p,j+1)\big((4j+2)p+4j+3\big)2^j.
\end{equation*}
By Lemmas \ref{lem: (4j+3)C_j} and \ref{lem: sum} it suffices to
show that
\begin{align*}
	\mathcal{A}_{p-1} 
	&\equiv \sum_{j=0}^{p-2}\binom{2j}j(j+1)(-2)^j\\[6pt]
	&\equiv \sum_{j=0}^{p-2}\binom{2j}j\ j\ (-2)^j\ +\ \sum_{j=0}^{p-2}\binom{2j}j(-2)^j\\[6pt]
	&\equiv -\frac49+1 \equiv \frac59\pmod{p}
\end{align*}
and
\begin{align*}
	\mathcal{B}_{p} 
	&\equiv 2+\sum_{j=0}^{p-2}C_j(4j+3)\left(\frac1{j+1} \binom{p-1}{j}\binom{p+j+1}{j}\right)2^j\\
	&\equiv 2+\sum_{j=0}^{p-2}C_j(4j+3)\left(\frac1{j+1} \binom{p-1}{j}\binom{j+1}{j}\right)2^j\\
	&= 2+\sum_{j=0}^{p-2}\left(2\binom{2j}j+\binom{2j+1}{j+1}\right)(-2)^j\\
	&\equiv 4\pmod{p}\text{.}
\end{align*}
Thus, we have 
\begin{equation*}
\sum_{k=1}^{p-1}(-1)^kk^2D_kD_{k-1}= \frac23p(p-1)\mathcal{A}_{p-1} \equiv -\frac23p\mathcal{A}_{p-1} \equiv\ -\frac56p\pmod{p^2}
\end{equation*}
and
\begin{equation*}
\sum_{k=1}^p (-1)^k(4k^2+2k-1)D_{k-1}s_k=-p\mathcal{B}_p \equiv\ -4p\pmod{p^2}.
\end{equation*}
The proof of Theorem \ref{th:2022sun Remark 5.3(2)} is now complete.
\qed

\section{Proof of Theorem \ref{th:2022sun}}
From the properties of $s_{n}(x)$, we can obtain the following result.
\begin{lem}\label{lem:x(1+x)}
	For any $n\in\bZ^+$, we have
	\begin{equation*}
		\frac{s_{n+1}(x)-(1+2x)s_n(x)}{x^2+x}\in\bZ[x] \quad and \quad	\frac{s_{n}(x)-(1+2x)s_{n+1}(x)}{x^2+x}\in\bZ[x].
	\end{equation*}
\end{lem}

\noindent\emph{Proof.}
When $n=1$, $s_2(x)-(1+2x)s_1(x)=0$.
For $n\geq2$, 
\begin{align*}
	&s_{n+1}(x)-(1+2x)s_{n}(x)\\
	=&\sum_{k=1}^{n+1}N(n+1,k)x^{k-1}(x+1)^{n+1-k}-(1+2x)\sum_{k=1}^{n}N(n,k)x^{k-1}(x+1)^{n-k}\\
	=&(x+1)^n-(1+2x)(x+1)^{n-1}+x^n-(1+2x)x^{n-1}\\
	&+\sum_{k=2}^{n}N(n+1,k)x^{k-1}(x+1)^{n+1-k}-(1+2x)\sum_{k=2}^{n-1}N(n,k)x^{k-1}(x+1)^{n-k}\\
	=&(x^2+x)\left( \sum_{k=2}^{n}N(n+1,k)x^{k-2}(x+1)^{n-k}-(1+2x)\sum_{k=2}^{n-1}N(n,k)x^{k-2}(x+1)^{n-1-k}\right)\\ &-(x^2+x)\Bigl( x^{n-2}+(x+1)^{n-2}\Bigr)
\end{align*}
Similarly, when $n=1$, $s_{1}(x)-(1+2x)s_{2}(x)=-4(x^2+x)$. For $n\geq2$, 
\begin{align*}
	&s_{n}-(1+2x)s_{n+1}\\
	=&\sum_{k=1}^{n}N(n,k)x^{k-1}(x+1)^{n-k}-(1+2x)\sum_{k=1}^{n+1}N(n+1,k)x^{k-1}(x+1)^{n+1-k}\\
	=&(x+1)^{n-1}-(1+2x)(x+1)^{n}+x^{n-1}-(1+2x)x^{n}\\
	&+\sum_{k=2}^{n-1}N(n,k)x^{k-1}(x+1)^{n-k}-(1+2x)\sum_{k=2}^{n}N(n+1,k)x^{k-1}(x+1)^{n+1-k}\\
	=&(x^2+x)\left( \sum_{k=2}^{n-1}N(n,k)x^{k-2}(x+1)^{n-1-k}-(1+2x)\sum_{k=2}^{n}N(n+1,k)x^{k-2}(x+1)^{n-k}\right)\\ &-(x^2+x)\Big((2x+1)x^{n-2}+(2x+3)(x+1)^{n-2}\Big) .
\end{align*}
The proof of Lemma \ref{lem:x(1+x)} is now complete.
\qed

 In 2018, Sun\cite{Sun2018} showed
\begin{equation}\label{eq: (1+2x)^3}
	s_{n+1}(x)=\sum\limits_{k=0}^{\lfloor\frac{n}{2} \rfloor}\binom{n}{2k}C_{k}(2x+1)^{n-2k}(x(x+1))^k.
\end{equation}
Then we obtain the following lemma.
\begin{lem}\label{lem: (1+2x)^{3}}
	For any even integer $n\in\bZ^+$, we have
	\begin{equation*}
		\frac{(1+2x)(2+n)s_{n+1}(x)^2+s_{n}(x)s_{n+1}(x)}{(1+2x)^3}\in\bZ[x].
	\end{equation*}
\end{lem}
\noindent\emph{Proof.}
When $n\in\bZ^+$ is a even, we clearly have
\begin{equation*}
	s_{n+1}(x)=\sum\limits_{k=0}^{\frac{n}{2}}\binom{n}{2k}C_{k}(2x+1)^{n-2k}(x(x+1))^k
\end{equation*}
and
\begin{equation*}
	s_{n}(x)=\sum\limits_{k=0}^{\frac{n-2}{2}}\binom{n-1}{2k}C_{k}(2x+1)^{n-1-2k}(x(x+1))^k
\end{equation*}
by \eqref{eq: (1+2x)^3}.
Therefore, we have
\begin{align*}
	&(1+2x)(2+n)s_{n+1}(x)^2+s_{n}(x)s_{n+1}(x) \\
	&\equiv (1+2x)(n+2)C_{\frac{n}{2}}^{2}(x^2+x)^n+(1+2x)(n-1)C_{\frac{n}{2}}C_{\frac{n-2}{2}}(x^2+x)^{n-1}\\
	&\equiv (1+2x)(x^2+x)^{n-1}C_{\frac{n}{2}}\Big( (n+2)(x^2+x)C_{\frac{n}{2}}+(n-1)C_{\frac{n-2}{2}}\Big)\\
	&\equiv 2(1+2x)(x^2+x)^{n-1}C_{\frac{n}{2}}\Bigg((x^2+x)\binom{n}{\frac{n}{2}}+(n-1)\frac{\binom{n-2}{\frac{n-2}{2}}}{n}\Bigg)\\
	&\equiv 2(1+2x)(x^2+x)^{n-1}C_{\frac{n}{2}}\binom{n}{\frac{n}{2}}(x^2+x+\frac{1}{4})\\
	&\equiv (1+2x)^3(x^2+x)^{n-1}C_{\frac{n}{2}}\binom{n-1}{\frac{n-2}{2}}\equiv 0 \pmod {(1+2x)^3}.
\end{align*}
This yields the desired Lemma \ref{lem: (1+2x)^{3}}.
\qed

\noindent\emph{Proof of Theorem \ref{th:2022sun}.}
Let $a_{k}=s_k(x)$ and $b_{k}=s_{k+1}(x)$. By Zeilberger’s algorithm\cite{Zeilberger1990}, we find that
\begin{equation}\label{ann:s_{k}(x)}
	(k+3)s_{k+2}(x)=(3+2k)(1+2x)s_{k+1}(x)-ks_{k}(x).
\end{equation}
By the telescoping method\cite{Hou2023}, we find that
\begin{equation*}
	-4(x^2+x)k(k+1)(k+2)a_{k}b_{k}=\Delta_{k}F_{k}
\end{equation*}
where
\begin{align*}
	F_{k}=\frac{1}{2}k(k+1)\Bigl(&(k+4)(k-1)a_{k-1}b_{k-1}-(1+2x)(k^2+7k+4)a_{k}b_{k-1}\notag\\
	&-(1+2x)(k+2)(k-1)a_{k-1}b_{k}+(k+2)(k+3)a_{k}b_{k}\Bigr).
\end{align*}
As $F_1=0$, we derive
\begin{align}\label{eq: 4/(n+1)(n+2)}
	&\frac{4}{(n+1)(n+2)}\sum_{k=1}^{n}(x^2+x)k(k+1)(k+2)s_k(x)s_{k+1}(x)\notag\\
	=-\frac{1}{2}\Bigg(&s_{n+1}(x)\Big( (n^{2}+7n+12)s_{n+2}(x)-(n^2+9n+12)(1+2x)s_{n+1}(x)\Big)\notag\\
	&+ns_{n}(x)\Big((5+n)s_{n+1}(x)-(n+3)(1+2x)s_{n+2}(x)\Big)  \Bigg)\notag\\
	=-\frac{1}{2}\Bigg(&s_{n+1}(x)\Big( n(n+7)s_{n+2}(x)-n(n+9)(1+2x)s_{n+1}(x)\notag\\
	&+12\big(s_{n+2}(x)-(1+2x)s_{n+1}(x)\big)\Big)\notag\\
	&+ns_{n}(x)\Big((5+n)s_{n+1}(x)-(n+3)(1+2x)s_{n+2}(x)\Big)\Bigg)
\end{align}
By \eqref{ann:s_{k}(x)}, we have
\begin{equation*}
	12\Big( s_{n+2}(x)-(1+2x)s_{n+1}(x)\Big)=n\Big(8(1+2x)s_{n+1}(x)-4s_{n+2}(x)-4s_{n}(x) \Big) .
\end{equation*}
Thus, by \eqref{eq: 4/(n+1)(n+2)} we deduce that
\begin{align}\label{eq: s_{k}(x)s_{k+1}(x)1}
	&\frac{4}{n(n+1)(n+2)}\sum_{k=1}^{n}(x^2+x)k(k+1)(k+2)s_k(x)s_{k+1}(x)\notag\\
	=-\frac{1}{2}\Bigg(&s_{n+1}(x)\Big( (n+3)s_{n+2}(x)-(n+1)(1+2x)s_{n+1}(x)-4s_{n}(x)\Big)\notag\\
	&+s_{n}(x)\Big((5+n)s_{n+1}(x)-(n+3)(1+2x)s_{n+2}(x)\Big)  \Bigg)\notag\\
	=-\frac{1}{2}\Bigg(&(n+3)s_{n+2}(x)\Big((s_{n+1}(x)-(1+2x)s_n(x)\Big)\notag\\
	&+(n+1)s_{n+1}(x)\Big(s_{n}(x)-(1+2x)s_{n+1}(x)\Big)\Bigg).
\end{align}
	  By combining \eqref{ann:s_{k}(x)} with \eqref{eq: s_{k}(x)s_{k+1}(x)1}, we  also derive
	  \begin{align}\label{eq: s_{k}(x)s_{k+1}(x)2}
	  	&\frac{4}{n(n+1)(n+2)}\sum_{k=1}^{n}(x^2+x)k(k+1)(k+2)s_k(x)s_{k+1}(x)\notag\\
	  	=-\frac{1}{2}\Bigg(&\Bigl((3+2n)(1+2x)s_{n+1}(x)-ns_{n}(x)\Bigr)\Bigl(s_{n+1}(x)-(1+2x)s_n(x)\Bigr)\notag\\
	  	&+(n+1)s_{n+1}(x)\Bigl(s_{n}(x)-(1+2x)s_{n+1}(x)\Bigr)\Bigg)\notag\\
	  	=-\frac{1}{2}\Big(&(1+2x)(2+n)s_{n+1}(x)^2+s_{n}(x)s_{n+1}(x)\notag\\
	  	&-(3+2n)(1+2x)^2s_n(x)s_{n+1}(x)+n(1+2x)s_n(x)^2\Big).
	  \end{align}
	  When $n\in\bZ^+$ is even, combining equation \eqref{eq: s_{k}(x)s_{k+1}(x)1} and \eqref{eq: s_{k}(x)s_{k+1}(x)2}  with Lemma \ref{lem:x(1+x)}, we clearly have
	  \begin{equation*}
	  	\frac{4}{n(n+1)(n+2)}\sum_{k=1}^{n}k(k+1)(k+2)s_k(x)s_{k+1}(x)\in\bZ[x].
	  \end{equation*}
 As $s_{2n}(x)/(1+2x)\in\bZ[x]$, we have
	 	\begin{equation*}
	 	\frac{-(3+2n)(1+2x)^2s_n(x)s_{n+1}(x)+n(1+2x)s_n(x)^2}{(1+2x)^3}\in\bZ[x]
	 \end{equation*}
	 when $n$ is even.

Since gcd$\big((x^2+x),(1+2x)^3\big)=1$, this completes the proof of Theorem \ref{th:2022sun} combining identity \eqref{eq: s_{k}(x)s_{k+1}(x)2} with Lemma \ref{lem: (1+2x)^{3}}.
\qed

\section* {Acknowledgements}
         
\quad The authors would like to thank Prof. Zhi-Wei Sun and Rong-Hua Wang for helpful comments on this paper. This work was supported by the Natural Science Foundation of Tianjin, China (No.
22JCQNJC00440).

\end{document}